\documentclass[reqno,12pt]{amsart}
\usepackage{amsmath, amssymb, amsthm, epsfig}

\usepackage{color}
\usepackage[inner=1.0in,outer=1.0in,bottom=1.0in, top=1.0in]{geometry}

\usepackage{hyperref, latexsym}
\usepackage{url}

\usepackage{color}

\def\today{\ifcase\month\or
  January\or February\or March\or April\or May\or June\or
  July\or August\or September\or October\or November\or December\fi
  \space\number\day, \number\year}

 \newtheorem{theorem}{Theorem}
 \newtheorem{lemma}{Lemma}
 \newtheorem{Proposition}{Proposition}
 
 \newtheorem{corollary}[theorem]{Corollary}
 \theoremstyle{definition}

 \theoremstyle{remark}

\newcommand{\Rep}{\textrm{Re}}
\newcommand{\hW}{\widehat{W}}

\newcommand{\hP}{\widehat{\Phi}}
\newcommand{\sumstar}{\sideset{}{^*}\sum}
\newcommand{\shortmod}{\ensuremath{\negthickspace \negthickspace
\negthickspace \pmod}}

\begin{document}

\title[Simple zeros and the asymptotic large sieve]{Simple zeros of
primitive Dirichlet $L$-functions and the asymptotic large sieve}

\date{\today}

\author[V. Chandee]{Vorrapan Chandee}
\address{Department of Mathematics \\ Burapha University \\ Chonburi, Thailand 20131; Centre de Recherches Mathematiques Universite de Montreal P.O. Box 6128 \\
Centre-Ville Station Montreal \\ Quebec H3C 3J7 
}
\email{vorrapan@buu.ac.th; chandee@crm.umontreal.ca}

\author[Y. Lee]{Yoonbok Lee}
\address{Department of Mathematics \\ University of Rochester \\ Rochester
\\ NY 14627 \\ USA}
\email{lee@math.rochester.edu}

\author[S. Liu]{Sheng-Chi Liu}
\address{Department of Mathematics \\ Mailstop 3368 \\ Texas A\&M University \\ College Station \\ TX 77843-3368}
\email{scliu@math.tamu.edu}

\author[M. Radziwi\l\l]{Maksym Radziwi\l\l}
\address{Department of Mathematics \\ Stanford University \\
450 Serra Mall, Bldg. 380 \\ Stanford, CA 94305-2125}
\email{maksym@stanford.edu}
\thanks{The first author is supported by a CRM-ISM fellowship.}
\thanks{The fourth author is partially supported by a NSERC PGS-D award.}

\begin{abstract}
Assuming the Generalized Riemann Hypothesis (GRH), we show using the asymptotic large sieve
that 91\% of the zeros of primitive Dirichlet $L$-functions are simple. This improves on earlier
work of \"{O}zl\"{u}k  which gives a proportion of at most 86\%. We further compute $q$-analogue
of the Pair Correlation Function $F(\alpha)$ averaged over all primitive Dirichlet $L$-functions in the
range $|\alpha| < 2$ .
Previously such a result was available only when the average included all the characters $\chi$. As a corollary of our results, we obtain an asymptotic formula for a sum over characters similar to the one encountered in the Barban-Davenport-Halberstam Theorem.
\end{abstract}

\allowdisplaybreaks
\numberwithin{equation}{section}

\maketitle

\section{Introduction}

Montgomery \cite{Montgomery} was the first to consider the Pair Correlation of the zeros of the Riemann zeta-function.
Montgomery's results suggested that the distribution of the zeros of the Riemann zeta-function follows the same laws as the distribution of the eigenvalues of a random unitary matrix. This connection was further expanded on, and is responsible for much of the subsequent activity in the theory of $L$-functions (see for example \cite{KatzSarnak}, \cite{KeatingSnaith}, \cite{RudnickSarnak}).

One can similarly investigate the distribution of the low-lying zeros in a family of $L$-functions.
\"{O}zl\"{u}k  \cite{Ozluk} considered a $q$-analogue of Montgomery's results. His motivation was to understand the low-lying zeros of $L(s,\chi)$ on average over $\chi$ modulo $q$ and
$Q \leq q \leq 2Q$. Since the family is larger, one can obtain better results than in the case of the Riemann zeta-function.

One defect in \"{O}zl\"{u}k's work was that it concerns an average over all characters $\chi$ rather than
just the primitive characters $\chi$. As a result, in applications this often leads to inferior results.

Recently, Conrey, Soundararajan, and Iwaniec developed an \textit{asymptotic large sieve} \cite{ConreyIwaniecSoundararajan}. They devised a method to obtain asymptotic estimates for rather general averages over primitive characters. In this paper we revisit \"{O}zl\"{u}k's work in the light of these recent developments, obtaining results for primitive characters rather than all characters. As a consequence we obtain that, in a suitable sense, 91\% of the zeros of primitive Dirichlet $L$-functions
are simple, on the assumption of the Generalized Riemann Hypothesis (GRH). Throughout this paper GRH is assumed.

Let $\Phi$ be a smooth function which is real and compactly supported in $(a,b)$ with $0< a< b$,
and define its Mellin transform
$$
\hP(s) = \int_{0}^{\infty} \Phi(x) x^{s - 1} \> dx.
$$
Let
$$
N_{\Phi}(Q) := \sum_{q} \frac{W(q/Q)}{\varphi(q)} {\sumstar_{\chi \shortmod{q}}}
\sum_{\gamma_{\chi}}
|\hP (i\gamma_{\chi})|^2
$$
with $W$ a smooth function, compactly supported in $(1,2)$,
the second sum being over primitive characters $\chi$, and
the last sum being over all non-trivial zeros $\tfrac 12 + i \gamma_\chi $ of Dirichlet $L$-function $L(s, \chi)$.
As we will see later (in Lemma \ref{counting})
$$
N_{\Phi}(Q) \sim \frac{A}{2\pi} Q \log Q \int_{-\infty}^{\infty} |\hP (i x)|^2 \> dx
$$
where
\begin{equation}\label{def:valueA}
A = \hW (1) \prod_{p} \bigg ( 1 - \frac{1}{p^2} - \frac{1}{p^3} \bigg ) .
\end{equation}
Our work yields the following theorem.
\begin{theorem}\label{cor:1}
Assume GRH.
The proportion of simple zeros of all primitive Dirichlet $L$-functions
is greater than or equal to $\tfrac {11}{12}$ in the sense of the inequality
\begin{equation*}
\frac{1}{N_{\Phi}(Q)} \sum_{q} \frac{W(q/Q)}{\varphi(q)} {\sumstar_{\chi \shortmod{q}}}
\sum_{\substack{\gamma_{\chi}\\ \text{simple}}} |\hP (i \gamma_{\chi})|^2 \geq
\frac{11}{12} + o(1)
\end{equation*}
with the sum being over primitive characters and with $\Phi$ chosen so that
$\hP (i x) = (\sin x / x)^2$.
\end{theorem}

We note that the function $\Phi$ satisfying $ \hP (ix) = ( \sin x / x )^2$ is not smooth, but we can still apply Theorem \ref{thm:1} to $\Phi$ since the condition $ \hP ( ix) \ll |x|^{-2}$ is good enough in our proof and can replace the smoothness.

\"{O}zl\"{u}k  obtains a similar lower bound but for \textit{all} Dirichlet $L$-functions rather than just the
primitive $L$-functions. This yields an over-count and as a result \"{O}zl\"{u}k's method is only capable
of delivering a proportion of $0.8688 \ldots$ simple zeros. This should be compared with our
proportion $\tfrac{11}{12} = 0.917 \ldots$. We will explain the number $0.8688 \ldots$ in Section \ref{8688}.

Following \"{O}zl\"{u}k  we consider the $q$-analogue of the Pair Correlation Function, which is defined as
$$ F_{\Phi}(Q^{\alpha};W) = \frac{1}{N_\Phi (Q)}  \sum_{q} \frac{W(q/Q)}{\varphi(q)} {\sumstar_{\chi
\shortmod{q}} } \bigg| \sum_{\gamma_\chi } \hP
\left( i\gamma_\chi \right)Q^{i\gamma_\chi \alpha}\bigg|^2.$$
Our main result is the following.
\begin{theorem}\label{thm:1}
Assume GRH. Let $\epsilon > 0$. Then
\begin{align*}
  F_\Phi (Q^\alpha; W) =& (1+o(1)) \left(  f(\alpha) + \Phi(Q^{-|\alpha|} )^2 \log Q \left( \frac{1}{ 2 \pi }  \int_{-\infty}^\infty \left| \hP ( ix ) \right|^2 dx    \right)^{-1}  \right) \\
  & + O(  \Phi( Q^{- |\alpha|} )  \sqrt{  f(\alpha) \log Q } )
\end{align*}
holds uniformly for $ |\alpha| \leq 2-\epsilon$ as $ Q \to \infty$, where
\begin{equation}\label{eqn:f alpha}
  f(\alpha ) :=  \begin{cases}
| \alpha | & \text{ for }  | \alpha |\leq 1, \\
1 & \text{ for }   | \alpha | >1 .
\end{cases}
\end{equation}
\end{theorem}

Since primitive Dirichlet $L$-functions form a unitary family, we conjecture
that for $\alpha \geq 2$ the same asymptotic formula continues to hold.
We obtain Theorem \ref{thm:1} by applying the asymptotic large sieve.
The proof of Theorem \ref{thm:1} starts with the explicit formula, for $X \geq 1 $
\begin{multline*}
\sum_{\gamma} \hP (i \gamma) X^{i\gamma}
= E(\chi) \hP (\tfrac 12) X^{1/2} - \sum_{n = 1}^{\infty}
\frac{\Lambda(n)\chi(n) \Phi  (  {n}/{X}   )   }{\sqrt{n}}
\\
+ \Phi ( X^{-1}  ) \log \frac{q}{\pi}
+ O\big(\min(X^{1/2}, X^{-1/2} \log q \log (1+X))\big),
\end{multline*}
where $E(\chi) = 0$ or $1$ according as $\chi \neq
\chi_0$ or $\chi = \chi_0$ and where $\tfrac 12 + i\gamma$
ranges over non-trivial zeros of $L(s,\chi)$. The term $ \Phi( X^{-1}  ) \log q/\pi$ contributes only when $X$ is small. Thus, Theorem \ref{thm:1} is essentially equivalent to the following proposition.
\begin{Proposition}\label{prop:1}
Assume GRH. Let $\epsilon > 0 $ and $X = Q^{\alpha}$. Then
$$
  \sum_{q} \frac{W(q/Q)}{\varphi(q)} {\sumstar_{\chi \shortmod{q}} } \left| \sum_n  \frac{\Lambda(n) \chi(n) \Phi \left( n/X \right)}{n^{1/2}} \right|^2
\sim f(\alpha)  N_{\Phi}(Q)
$$
uniformly for $  | \alpha | \leq 2 - \epsilon $ as $Q \rightarrow \infty$, where $f(\alpha)$ is defined in \eqref{eqn:f alpha}.
\end{Proposition}

By applying Lemma \ref{counting} to Proposition \ref{prop:1} and letting $\Phi_0(u) = \Phi(u) / \sqrt u$ and so $ \hP (ix ) = \hP_0 ( \tfrac 12+ix) $, we have the following corollary.

\begin{corollary}\label{colofprop1}
Assume GRH. Let $\epsilon > 0 $ and $X = Q^{\alpha}$. Then
\begin{align*}
\sum_{q} \frac{W(q/Q)}{\varphi(q)}  {\sumstar_{\chi \shortmod{q}} } \left| \sum_n \Lambda(n) \chi(n) \Phi_0 \left( n/X \right)\right|^2
 \sim   f(\alpha) \frac{A}{2\pi} XQ \log Q \int_{-\infty}^{\infty} \left| \hat{\Phi}_0 \left(\tfrac12 + ix\right)\right|^2 \> dx
\end{align*}
uniformly for $  | \alpha | \leq 2 - \epsilon $, where $f(\alpha)$ is defined in \eqref{eqn:f alpha} and $A$ is defined in \eqref{def:valueA}.
\end{corollary}
Corollary \ref{colofprop1} is similar to Barban-Davenport-Halberstam Theorem in the range $Q \geq X^{1/2 + \epsilon}.$ We will discuss this connection in Section \ref{sec:BDH}.

The deduction of Theorem \ref{cor:1} from Theorem \ref{thm:1} can be found in Section 6 of \"{O}zl\"{u}k's paper, but we reproduce it in Section \ref{proof of theorem 1} for completeness.
The remainder of this paper is devoted to the proof of
Proposition \ref{prop:1}.

The bulk of the proof of Proposition \ref{prop:1} is the estimation of the contribution of the off-diagonal terms.
When $\alpha > 1$ we extract an additional main term from the terms
with $|m - n| \asymp Q$.

\section{Lemmas}

As announced in the introduction we start out by evaluating asymptotically
$N_{\Phi}(Q)$.
\begin{lemma}\label{counting} Assume GRH. We have,
$$
N_{\Phi}(Q) \sim \frac{A}{2\pi} Q \log Q \int_{-\infty}^{\infty}
|\hP (i x)|^2 d x
$$
as $Q \rightarrow \infty$, with
$$
A = \hW (1) \prod_{p} \bigg ( 1 - \frac{1}{p^2} -
\frac{1}{p^3} \bigg ).
$$
\end{lemma}
\begin{proof}
Let $N(\chi,T)$ denote the number of zeros of $L(s,\chi)$ in the
rectangle $0 < \sigma < 1$ and $-T \leq t \leq T$. It is
a well-known fact (see \cite{Selberg}) that if the conductor of $\chi$ is $q$, then
$$
N(\chi,T) = \frac{ T}{\pi} \log \frac{q T}{2\pi e} + O
\left ( \frac{\log (q T)}{\log\log (q T + 3)} \right ).
$$
uniformly in $q T > 1$. Integrating by parts we find
\begin{align*}
\sum_{\gamma_{\chi}} |\hP (i \gamma_{\chi})|^2
& = \int_{0}^{\infty} |\hP ( i t)|^2  dN(\chi, t)
 = \frac{1}{\pi}  \log q \int_{0}^{\infty}
|\hP (i x)|^2 d x + O \bigg ( \frac{\log q}{\log\log q} \bigg ) \\
& = \frac{\log q}{2\pi} \int_{-\infty}^{\infty} |\hP(i x)|^2 dx
+ O \bigg ( \frac{\log q}{\log\log q} \bigg )
\end{align*}
and
\begin{align*}
N_{\Phi}(Q) & := \sum_{q} \frac{W(q/Q)}{\varphi(q)} \sumstar_{\chi
\shortmod{q}} \sum_{\gamma_{\chi}} |\hP(i \gamma_{\chi})|^2
\\
& = \int_{-\infty}^{\infty} |\hP(i x)|^2 dx
\cdot \sum_{q} \frac{W(q/Q)}{\varphi(q)} \cdot \frac{\log q}{2\pi}
\cdot \varphi^{*}(q) + O\left(  \frac{ Q \log Q }{ \log\log Q } \right)
\end{align*}
where $\varphi^{*}(q) = \sum_{cd = q} \varphi(d) \mu(c)$ is the number of
primitive characters modulo $q$. Since $W$ is compactly supported in $(1,2)$
we have $\log q = \log Q + O(1)$ in the summation. Therefore,
\begin{equation} \label{equation1}
N_{\Phi}(Q) \sim \frac{\log Q}{2\pi}\int_{-\infty}^{\infty} |\hP(i x)|^2
dx  \sum_{q} W(q/Q) \frac{\varphi^{*}(q)}{\varphi(q)}.
\end{equation}
Since $\varphi^*$ and $\varphi$ are multiplicative, we have
\begin{align}\label{eqn:varphisum}
\sum_{\substack{q}}\frac{\varphi^*(q)}{\varphi(q)q^s} &= \prod_p \left( 1 + \frac{\varphi^*(p)}{\varphi(p)p^s} + \frac{\varphi^*(p^2)}{\varphi(p^2)p^{2s}} + \cdots \right) \nonumber \\
&= \zeta(s) \prod_p \left( 1 - \frac{1}{(p-1)p^s} + \frac{1}{(p-1)p^{2s}} - \frac{1}{p^{2s + 1}}\right) = \zeta(s) g(s),
\end{align}
where $g(s)$ is absolutely convergent for $\Rep (s) > 0$ and bounded on $ \Rep(s) \geq \varepsilon $ for any $\varepsilon>0$.
Using Mellin inversion formula,
$$
W(x) = \frac{1}{2\pi i} \int_{(c)} \hW (s) x^{-s} ds, \quad c>1
$$
and (\ref{eqn:varphisum}),
we obtain that
\begin{equation} \label{equation2}
\sum_{q} W(q/Q) \frac{\varphi^{*}(q)}{\varphi(q)} =
\frac{1}{2\pi i}\int_{(c)} \hW (s) \zeta(s)g(s) Q^{s} ds
= \hW (1) g(1) Q + O(Q^{\varepsilon})
\end{equation}
by shifting the contour to $\Rep(s) = \varepsilon$. Combining (\ref{equation1}) and (\ref{equation2}) we conclude that
\begin{equation*}
N_{\Phi}(Q)  \sim \frac{Q \log Q}{2\pi}  \hW (1) g(1) \int_{-\infty}^{\infty}
|\hP(i x)|^2 dx
 = \frac{A Q \log Q}{2\pi} \int_{-\infty}^{\infty} |\hP(i x)|^2 dx.
\end{equation*}

\end{proof}
The next four lemmas correspond to estimates of various
types of prime sums. The proofs are standard, but we present them here
for completeness.
\begin{lemma}\label{lem:p sum}
We have
\begin{equation*}
 \sum_{n  } \frac{\Lambda(n) \chi(n) \Phi  (  {n}/{X}   )  }{\sqrt{n}}  = \sum_{p } \frac{\Lambda(p) \chi(p) \Phi  ( p/X  )  }{\sqrt{p}}    + O(1).
\end{equation*}
\end{lemma}
\begin{proof}
By splitting the sum into three cases $ n=p$, $n=p^2 $ and $ n=p^k$ with $k > 2$, we get
\begin{equation*}
 \sum_{n  } \frac{\Lambda(n) \chi(n) \Phi  (  {n}/{X}   )  }{\sqrt{n}}  = \sum_{p } \frac{\Lambda(p) \chi(p)\Phi  ( p/X  )   }{\sqrt{p}} + \sum_{p } \frac{\Lambda(p^2) \chi(p^2)\Phi  ( p^2 /X  )   }{p}   + O(1).
\end{equation*}
Since $ \Phi$ has a compact support in $(a,b)$ for some $ 0 < a < b$, the last sum is
\begin{align*}
\left| \sum_{p } \frac{\Lambda(p^2) \chi(p^2)\Phi  ( p^2 /X  )  }{p}    \right| & \ll \sum_{ \sqrt{aX} < p < \sqrt{bX} }\frac{ \log p }{p} \\
& = \log \sqrt{bX} - \log \sqrt{aX} + O(1 ) \\
& = O(1).
\end{align*}
 Hence, we prove the lemma.
\end{proof}

\begin{lemma}\label{lem:psum2}
As $ X \to \infty,$
$$  \sum_{p} \frac{\log^2 p  \ \Phi ^2 ( p  /X  )  }{p}  =\frac{1}{2 \pi} \int_{- \infty}^\infty  |\hP(it)|^2 dt  \log X  +O(1).$$
\end{lemma}
\begin{proof}
Note that
$$ \sum_{p} \frac{\log^2 p \ \Phi ^2 ( p  /X  ) }{p}   = \sum_n \frac{\Lambda (n) \log n  \ \Phi ^2 (n  /X  )  }{n}      +O(1). $$
By Mellin inversion
$$ \sum_n \frac{\Lambda (n) \log n   \ \Phi ^2 ( n  /X  )  }{n}  = \frac{1}{ ( 2 \pi i )^2 } \int_{(c_1)} \int_{(c_2)}   \hP ( s_1 ) \hP ( s_2 ) \sum_n \frac{\Lambda (n) \log n}{n^{1+s_1+s_2 } } X^{s_1 + s_2 } ds_2 ds_1  $$
 for $ c_1 , c_2 > 0$. Since $ \sum_n \Lambda(n) (\log n ) n^{-s} = ( \zeta'/\zeta)'(s) $, the above integral equals
\begin{equation} \label{integral}
\frac{1}{ ( 2 \pi i )^2 } \int_{(c_1)} \int_{(c_2)} \hP ( s_1 ) \hP ( s_2 )  ( \zeta' / \zeta )' ( 1+ s_1 + s_2 ) X^{s_1 + s_2 } ds_2 ds_1 .
\end{equation}
By shifting the contour integral to $ \Rep ( s_2 ) = -c_1 - \varepsilon $,  we pick up a double pole at $ s_2 = -s_1 $. Hence we have (\ref{integral}) equals
\begin{align*} \begin{split}
 & \frac{1}{  2 \pi i  } \int_{(c_1)}   \hP ( s_1 ) \hP ( -s_1 )  \log  X   ds_1 + O(1) \\
 & \qquad +  \frac{1}{ ( 2 \pi i )^2 } \int_{(c_1)} \int_{(-c_1- \varepsilon)} \hP ( s_1 ) \hP ( s_2 )  ( \zeta' / \zeta )' ( 1+ s_1 + s_2 ) X^{s_1 + s_2 } ds_2 ds_1.
\end{split}\end{align*}
For the first integral we shift the contour to $\Rep ( s_1 ) = 0 $ without passing any poles and it becomes
$$  \frac{\log X}{2 \pi} \int_{- \infty}^\infty \hP ( it ) \hP ( -it ) dt  = \frac{\log X}{2 \pi} \int_{- \infty}^\infty | \hP ( it ) |^2 dt   .$$
The double integral is easily bounded by $O(X^{-\varepsilon})$.

\end{proof}

\begin{lemma} \label{lem:sumnonprincChar} Assume GRH for $L(s, \Psi)$. If $\Psi$ is a principal character, then for any $\varepsilon > 0$
$$ \sum_{p   } \frac{ \Psi(p) \Phi\left( p/X \right)\log p}{\sqrt{p}} = \hP \left(  \tfrac12 \right)  \sqrt{X} + O( Q^{ \varepsilon}) .$$
 If $\Psi$ is not a principal character, then  for any $\varepsilon > 0$
$$\sum_{p   } \frac{ \Psi(p) \Phi\left( p/X \right)\log p}{\sqrt{p}} \ll_{\varepsilon} Q^{\varepsilon}.$$
\end{lemma}
\begin{proof}
By Mellin inversion of $\Phi$, we have
$$ \sum_p  \frac{ \Psi(p) \Phi\left( p/X \right)\log p}{\sqrt{p}}  = \frac{1}{2 \pi i } \int_{ ( c)}  \hP (s) X^s \sum_p  \frac{ \Psi(p)  \log p}{p^{1/2+s}} ds . $$
The sum over $p$ has an analytic continuation via
$$ \sum_p  \frac{ \Psi(p)  \log p}{p^{1/2+s}} = - \frac{L'}{L} ( \tfrac 12+ s, \Psi ) + G(s), $$
where $ G(s)$ is analytic in $ \Rep(s) > 0$ and is uniformly bounded for $ \Rep (s) \geq   \varepsilon > 0 $.  By moving the contour to $ \Rep(s) =   \varepsilon$, we can prove the lemma.
\end{proof}

\begin{lemma}\label{lem:primeSL}
For $ |\Rep (z) | \leq \varepsilon$ and $ \Rep (s) < 0$, we have
\begin{align*}
\sum_{ p } \frac{\log p  \ \Phi (p / X)}{p^{1/2+z}} B_{-s}(p)  R_{-s}(p) =   \hP ( \tfrac12 - z ) X^{ 1/2-z} + O( X^{ 2\varepsilon} \log ( 2+|z|) ),
 \end{align*}
where
\begin{align*}
 B_s(m) &= \prod_{p |m } \left( 1 - \frac{1}{p^{s + 1}}\right), \\
 R_s(m) &= \prod_{p | m} \left(1 + \frac{1}{(p-1)p^{s+1}}\right)^{-1}.
\end{align*}
\end{lemma}

\begin{proof}
  By Mellin inversion of $\Phi$, we have
\begin{align}\label{lem:eqn1}
\sum_{ p } \frac{\log p \ \Phi(p/X)}{p^{1/2+ z}} B_{-s}(p)  R_{-s}(p) = \frac{1}{ 2 \pi i }  \int_{(c )} \hP (w)  X^{w  } \sum_p  \frac{\log p \  B_{-s}(p)R_{-s}(p) }{p^{1/2+z+w }}   dw
\end{align}
for $ c > 1/2 +  \varepsilon$. Define a function $H(w,s)$ by
\begin{align*}
 H(w,s) &:= \frac{ \zeta'}{\zeta} ( w)  + \sum_p \frac{\log p  B_{-s}(p)R_{-s}(p) }{p^w}   \\
 & = \frac{ \zeta'}{\zeta} ( w) + \sum_p \frac{ \log p }{ p^w} \bigg( 1- \frac{1}{ p^{1-s} } \bigg) \bigg( 1+ \frac{1}{ (p-1) p^{1-s}} \bigg)^{-1} .
\end{align*}
 If $ \Rep(s) < 0$, then $H(w,s)$ is an analytic function of $w$ in $ \Rep(w) > \tfrac 12$ and bounded on $ \Rep (w) \geq \tfrac 12+ \varepsilon' > \tfrac 12 $. Applying this
identity to \eqref{lem:eqn1} and shifting the contour to $ 2 \varepsilon$, we
have
\begin{align*}
\sum_{ p } \frac{\log p \ \Phi(p/X)}{p^{1/2+ z}} B_{-s}(p)  R_{-s}(p)
  = &\frac{1}{ 2 \pi i }  \int_{(c )} \hP (w)  X^{w  } \bigg(  - \frac{ \zeta'}{\zeta} ( \tfrac12 + z +w )    +  H(\tfrac12 + z+w ,s) \bigg) dw  \\
=  & \ \hP ( \tfrac12 - z ) X^{ 1/2- z} + O( X^{ 2 \varepsilon} \log ( 2+|z|) ).
 \end{align*}
\end{proof}

The next lemma can be proved by changing the sum to its Euler product. The proof is quite standard and we omit it.
\begin{lemma} \label{lem:sumVarphid}
Suppose that $(a,m) = 1$. Then
$$\sum_{(d,m) = 1} \frac{1}{\varphi(ad) d^s} = \frac{1}{\varphi(a)} \zeta(1 + s) K(s) B_s(m) R_s(a)R_s(m),$$
where $B_s$ and $R_s$ are defined in Lemma \ref{lem:primeSL} and
\begin{align*}
 K(s) = \prod_{p} \left(1 + \frac{1}{(p-1)p^{s+1}}\right).
\end{align*}
\end{lemma}

\section{Proof of Proposition \ref{prop:1}}
 Proposition \ref{prop:1} is equivalent to
 \begin{equation*}
 S = \sum_{q} \frac{W(q/Q)}{\varphi(q)} {\sumstar_{\chi \shortmod{q} }}\left| \sum_{p } \frac{\log p ~\chi(p) \Phi (p/X) }{\sqrt{p}} \right|^2 \sim f(\alpha) N_{\Phi}(Q)
\end{equation*}
by Lemma \ref{lem:p sum}. For notational convenience we let
\begin{equation} \label{def:ap}
a_p = \frac{\log p\ \Phi (p/X)}{\sqrt{p}}
\end{equation}
and define
\begin{equation}\label{def:Deltapr}
 \Delta(p,r) = \sum_{\substack{q \\ (q, pr) = 1}} \frac{W(q/Q)}{\varphi(q)} {\sumstar_{\chi \shortmod{q}}} \chi(p) \overline{\chi}(r)
\end{equation}
for primes $p$ and $r$,  then we have
\begin{equation*}
S = \sum_{p, r} a_p a_r \Delta(p,r)   = \sum_{p} a_p^2 \Delta(p,p) + \sum_{\substack{p, r \\ p \neq r}} a_p a_r  \Delta(p,r) = S_D + S_{N},
\end{equation*}
where $S_D$ is the sum of diagonal terms and $S_N$ is the sum of non-diagonal terms.

\subsection{The diagonal term $S_D$}
By \eqref{def:Deltapr} and Mellin inversion, we obtain that
\begin{equation*}
\Delta(p,p) = \sum_{\substack{q \\ (q,p) = 1}}  \frac{\varphi^*(q)}{\varphi(q)} W\left( \frac{q}{Q}\right) = \frac{1}{2\pi i} \int_{(c)} \hW(s) \bigg( \sum_{\substack{q \\ (q,p) = 1}}\frac{\varphi^*(q)}{\varphi(q)q^s} \bigg) Q^s \> ds
\end{equation*}
for $ c>1$.
By applying \eqref{eqn:varphisum} and then shifting the contour to the line $\Rep (s) = \varepsilon >0$, we have
\begin{align*}
\Delta(p,p) &=  \frac{1}{2\pi i} \int_{(c)} \hW(s) \zeta(s) g(s)  \left( 1 - \frac{1}{p^s}\right) \left
( 1 - \frac{1}{(p-1)p^s} + \frac{1}{(p-1)p^{2s}} - \frac{1}{p^{2s + 1}} \right )^{-1}
Q^s \> ds  \\
&= \hW(1) g(1) \left( 1 - \frac{1}{p}\right) \left ( 1 - \frac{1}{(p-1)p} + \frac{1}{(p-1)p^{2}} - \frac{1}{p^{3}} \right )^{-1} Q + O(Q^{\varepsilon})  .
\end{align*}
By this equation and Lemma \ref{lem:psum2}, we then obtain
\begin{align*}
S_D =  & \hW(1) g(1) Q \sum_{p} \frac{\log^2 p\ \Phi^2 ( p / X  ) }{p}   + O(Q  )  \\
=& A Q \log X \frac{1}{2\pi} \int_{-\infty}^{\infty} |\hP (it)|^2 dt    + O(Q ),
\end{align*}
where $A$ is defined as in (\ref{def:valueA}).

\subsection{The non-diagonal term $S_N$}
We first observe that
\begin{align*}
\begin{split}
\Delta(p,r) &= \sum_{\substack{q \\ (q, pr) = 1}} \frac{W(q/Q)}{\varphi(q)} \sumstar_{\chi \shortmod{q}} \chi(p) \overline{\chi}(r)\\
&= \sum_{\substack{q \\ (q, pr) = 1}} \frac{W(q/Q)}{\varphi(q)} \sum_{\substack{d|q \\ d| p-r}} \varphi(d) \mu \left( \frac{q}{d}\right) \\
&= \sum_{\substack{d \\ d| p-r}} \varphi(d) \sum_{\substack{c \\ (cd, pr) = 1}} \frac{W(cd/Q) \mu(c)}{\varphi(cd)}
\end{split}
\end{align*}
for primes $p$ and $r$. We want to replace the condition $d|p-r$ by the character sum using
$$  \sum_{\Psi \shortmod{d}} \Psi(p) \overline{\Psi}(r) = \begin{cases}
\varphi(d) & \text{for  }  d|p-r, (pr, d) =1 , \\
0  & \text{otherwise }.
\end{cases}   $$
However, it is not effective in our application when $d$ is large. Hence
we introduce a new parameter $C$ and we split the above
sum according as $c \leq C$ or $c > C$ in order to handle the condition $d|p-r$ when $d$ is large. Thus we define
\begin{align}\label{def:U L}\begin{split}
 U(p,r) & = \sum_{\substack{d \\ d| p-r}} \varphi(d) \sum_{\substack{c > C \\ (cd, pr) = 1}} \frac{W(cd/Q) \mu(c)}{\varphi(cd)},\\
 L(p,r) & = \sum_{\substack{d \\ d| p-r}} \varphi(d) \sum_{\substack{c \leq C \\ (cd, pr) = 1}} \frac{W(cd/Q) \mu(c)}{\varphi(cd)},
\end{split}\end{align}
so that
$$ \Delta(p,r) = U(p,r) + L(p,r). $$
 Then by calculating the sums
\begin{align*}
 S_U &:= \sum_{p \neq r} a_p a_r U(p,r) ,\\
 S_L &:=  \sum_{p \neq r } a_p a_r L(p,r),
\end{align*}
we can evaluate the sum
$$
 S_N = S_U + S_L .
$$
Since $ W$ is supported in $(1,2)$, we have $ cd \asymp Q$. If $c > C$ then $d \ll Q / C$ and replacing the condition
$d | p - r$ by a character sum modulo $d$ in $U(p,r)$ is efficient and leads to good estimates for $S_U$. We perform
this computation in Section \ref{sec:SU}.

On the other hand, in the case $c \leq C$, we have  $d \gg Q / C$ and the above method using modulo $d$ character sums does not work. Instead we write $d e = |p - r|$, and replace
the condition $d | p - r$ by $e | p - r$ to eliminate $d$ from our sum
by expressing $d$ as $|p - r|/e$. Now we have
$e \ll X C / Q$ which is small enough, so that the modulo $e$ character sums replacing $e|p-r$ works well. This allows us to resume our argument in the case of the sum $S_L$. We consider the sum $S_L$ in Section \ref{sec:SL}. For a technical reason the above idea will be modified slightly.

\subsection{Evaluating $S_U$} \label{sec:SU}
We first consider the sum $U(p,r)$ defined in \eqref{def:U L}. Replacing the condition $ d|p-r $ by a character sum, we have
\begin{equation*}
U(p,r)  =  \sum_{\substack{c > C }} \mu(c) \sum_{\substack{d \\ (cd, pr) = 1}} \frac{W(cd/Q)}{\varphi(cd)} \sum_{\Psi \shortmod{d}} \Psi(p) \overline{\Psi}(r) .
\end{equation*}
We denote the sum corresponding to $\Psi = \Psi_0$ in $U(p,r)$ by
\begin{equation*}
U_0(p,r) = \sum_{\substack{c > C }} \mu(c) \sum_{\substack{d \\ (cd, pr) = 1}} \frac{W(cd/Q)}{\varphi(cd)}
\end{equation*}
and the others by
\begin{equation*}
U_E(p,r) = \sum_{\substack{c > C }} \mu(c) \sum_{\substack{d \\ (cd, pr) = 1}} \frac{W(cd/Q)}{\varphi(cd)} \sum_{\substack{\Psi \shortmod{d} \\ \Psi \neq \Psi_0}} \Psi(p) \overline{\Psi}(r).
\end{equation*}
Then $ U(p,r) = U_0 (p,r) + U_E (p,r)$ and $ S_U = S_{U_0} + S_{ U_E}$, where $S_{U_0} := \sum_{ p \neq r } a_p a_r U_0 ( p,r) $ and $ S_{U_E} := \sum_{p \neq r } a_p a_r U_E ( p,r) $.

We consider the sum $S_{U_0 }$. Since $ \sum_{ c|k } \mu (c) = 1 $ for $ k=1 $ and $0 $ for $k >1$, we have
\begin{align*}
 U_0(p,r) &= \sum_{\substack{c > C }} \mu(c) \sum_{\substack{d \\ (cd, pr) = 1}} \frac{W(cd/Q)}{\varphi(cd)} = \sum_{\substack{k \\ (k, pr) = 1}} \frac{W(k/Q)}{\varphi(k)} \sum_{\substack{c | k \\ c > C}} \mu(c) \\
& = - \sum_{\substack{k \\ (k, pr) = 1}} \frac{W(k/Q)}{\varphi(k)} \sum_{\substack{c | k \\ c \leq C}} \mu(c)
 = - \sum_{\substack{c \leq C }} \mu(c) \sum_{\substack{d \\ (cd, pr) = 1}} \frac{W(cd/Q)}{\varphi(cd)}.
\end{align*}
 By Mellin inversion, we have that
\begin{align*}
U_0(p,r) = -\sum_{\substack{c \leq C  \\ (c, pr) = 1}} \mu(c)  \frac{1}{2\pi i} \int_{(2)} \hW(s) \frac{Q^s}{c^s}  \sum_{\substack{d \\ (d, pr) = 1}} \frac{1}{\varphi(cd) d^s} ds.
\end{align*}
By Lemma \ref{lem:sumVarphid}, we obtain that
\begin{align*}
U_0(p,r) &= -\sum_{\substack{c \leq C  \\ (c, pr) = 1}} \frac{\mu(c)}{\varphi(c)} \frac{1}{2\pi i} \int_{(2)} \hW(s) \frac{Q^s}{c^s} \zeta(1 + s)K(s) B_s(pr) R_s(c) R_s(pr) \> ds .
\end{align*}
We move the contour integral to $ \Rep (s) = -1 + \varepsilon$ and pick up a simple pole at $s = 0.$ Then
\begin{align*}
U_0(p,r) &=  -\hW(0)K(0) B_0(pr) R_0(pr)\sum_{\substack{c \leq C  \\ (c, pr) = 1}} \frac{\mu(c)R_0(c)}{\varphi(c)} + O\left( \frac{C}{Q}Q^{\varepsilon}\right)
\end{align*}
and so
\begin{align}\label{SU0:eq1}
S_{U_0} = -\hW(0)K(0)\sum_{\substack{p, r \\ p \neq r}} a_pa_r B_0(pr) R_0(pr)\sum_{\substack{c \leq C  \\ (c, pr) = 1}} \frac{\mu(c)R_0(c)}{\varphi(c)} + O\bigg(  \bigg| \sum_{p} a_p \bigg|^2 \frac{C}{Q}Q^{\varepsilon}\bigg).
\end{align}
Now we evaluate the main term of $S_{U_0}$. The condition $ (c, pr)=1$ can be disregarded with an additional error term $ X^{1/2+\varepsilon}$. Now the sum over $p$ and $ r$ is
\begin{align}\label{SU0:eq2}
 \sum_{\substack{p, r \\ p \neq r}} a_pa_r B_0(p)B_0 (r) R_0(p)R_0 ( r)  & =  \sum_{p,r} a_pa_r B_0(p)B_0( r) R_0(p)R_0( r) - \sum_p a_p ^2 B_0 (p)^2 R_0 (p)^2  \nonumber \\
& = \bigg( \sum_p a_p B_0 (p) R_0 (p) \bigg)^2 + O( X^\varepsilon )
\end{align}
by adding and subtracting diagonal terms.  Similarly to Lemma \ref{lem:sumVarphid}, we can obtain
\begin{equation}\label{SU0:eq3}
 \sum_p a_p B_0 (p) R_0 (p) = \sqrt{X} \hP \left( \tfrac12 \right) + O( X^{\varepsilon}).
\end{equation}
Therefore, since the sum in the error term of \eqref{SU0:eq1} is $ \sum_p a_p \ll \sqrt{X} , $ we have
\begin{align} \label{eqn:sumUmain}
  S_{U_0} =- \hP(\tfrac12)^2    \hW(0)K(0) X  \sum_{ c \leq C    } \frac{\mu(c)R_0(c)}{\varphi(c)}   + O\left( X^{1/2+ \varepsilon}+  \frac{XC}{Q}Q^{\varepsilon}  \right)
\end{align}
 by \eqref{SU0:eq1}--\eqref{SU0:eq3}.

The next lemma shows the contribution of $U_E$ is small, so that we can conclude
\begin{align}\label{eqn:sumU}
S_U = S_{U_0 } + S_{U_E} = - \hP(\tfrac12)^2    \hW(0)K(0) X  \sum_{ c \leq C    } \frac{\mu(c)R_0(c)}{\varphi(c)}   + O\left(   X^{ 1/2 + \varepsilon}+ \frac{XC}{Q}Q^{\varepsilon} +  \frac{Q^{1 + \varepsilon}}{C}  \right)
\end{align}
by \eqref{eqn:sumUmain} and Lemma \ref{eqn:sumUerror}.

\begin{lemma} \label{eqn:sumUerror}
Assume GRH. We have
$$
S_{U_E} =  \sum_{\substack{p, r \\ p \neq r}}  a_p a_r U_E (p,r)   \ll \frac{Q^{1 + \varepsilon}}{C}.
$$
\end{lemma}

\begin{proof}
We write
\begin{align}\label{eqn:lemSUE}
S_{U_E} &= \sum_{\substack{p, r}} \frac{\log p \log r  \ \Phi (p/X) \Phi(r/X) }{\sqrt{pr}}   U_E(p,r) - \sum_{\substack{p}} \frac{\log^2 p \ \Phi^2 (p/X)    }{p} U_E(p,p).
\end{align}
By Lemma \ref{lem:sumnonprincChar}, the first sum in \eqref{eqn:lemSUE} is
\begin{align*}
\sum_{c > C} \mu(c)   \sum_d \frac{W(cd/Q)}{\varphi(cd)} &\sum_{\substack{\Psi \shortmod{d} \\ \Psi \neq \Psi_0}} \left| \sum_{ p \nmid cd }  \frac{ \Psi(p) \Phi\left( p/X\right)\log p}{\sqrt{p}} \right|^2 \\
& \ll \sum_{c > C} \frac{1}{\varphi(c)} \sum_{d \leq \tfrac{2Q}{c}} \frac{1}{\varphi(d)} \varphi(d) Q^{\varepsilon} \ll \frac{Q^{1 + \varepsilon}}{C}.
\end{align*}
The second sum in \eqref{eqn:lemSUE} is also bounded by
$$ \log^2 X \sum_{c > C} \frac{1}{\varphi(c)} \sum_{d \leq \tfrac{2Q}{c}} \frac{1}{\varphi(d)} \varphi(d) \ll \frac{Q^{1 + \varepsilon}}{C}$$
in a similar way, whence the lemma follows.
\end{proof}
\subsection{Evaluating $S_L$}\label{sec:SL}
Recall that
$$ L(p,r) = \sum_{c \leq C} \mu(c) \sum_{\substack{d | p - r \\ (cd, pr) = 1 }} \frac{W(cd/Q)}{\varphi(cd)} \varphi(d)$$
for primes $p$ and $r$. For distinct primes $p$ and $ r$, the condition  $d|p-r$ implies $ ( d,pr) = 1$. So we can erase the condition $(d,pr)=1$, getting
$$ L(p,r) = \sum_{ \substack{ c \leq C \\ (c, pr) = 1 }} \mu(c) \sum_{ d | p - r } \frac{W(cd/Q)}{\varphi(cd)} \varphi(d).$$
Using the identity
$$ \frac{\varphi(d)}{\varphi(cd)} = \frac{1}{\varphi(c)}\prod_{p | (d,c)} \left( 1 - \frac{1}{p}\right) = \frac{1}{\varphi(c)} \sum_{a|c, a | d} \frac{\mu(a)}{a},$$
we have
\begin{align*}
 L(p,r)  = \sum_{\substack{ c \leq C \\ (c, pr) = 1 } } \frac{\mu(c)}{\varphi(c)} \sum_{d | p - r  } W \left( \frac{cd}{Q} \right) \sum_{a|c, a | d} \frac{\mu(a)}{a}
  = \sum_{\substack{ c \leq C, a|c \\ (c, pr) = 1 } } \frac{\mu(a) \mu(c)}{a\varphi(c)} \sum_{ad | p - r  } W \left( \frac{acd}{Q} \right).
\end{align*}
Letting $ ade = |p-r|$, we change the sum over $d$ to the sum over $ e$ as follows
$$ L(p,r) = \sum_{\substack{ c \leq C, a|c \\ (c, pr) = 1 } } \frac{\mu(a) \mu(c)}{a\varphi(c)} \sum_{ae | p - r  } W \left( \frac{c|p-r| }{Qe} \right) . $$
Now we can replace the condition $ae| p-r $ by a character sum modulo $ae$, getting
\begin{align*}
 L(p,r) = \sum_{\substack{ c \leq C, a|c  \\ (c, pr) = 1 } } \frac{\mu(a) \mu(c)}{a\varphi(c)} \sum_e W \left( \frac{c|p-r| }{Qe} \right)  \frac{1}{ \varphi(ae) } \sum_{\Psi \shortmod{ae} } \Psi (p) \overline{\Psi}(r) .
\end{align*}
Similarly to the sum $U(p,r)$, we split the sum $L(p,r)$ into two parts $ L_0 (p,r)$ and $L_E (p,r) $, where $L_0(p,r)$ is the sum coming from the principal character
$\Psi = \Psi_0$ and $L_E(p,r)$ is the sum coming from the remaining
non-principal characters.

We compute the contribution from $L_0(p,r).$ Define
$$ S_{L_0} := \sum_{\substack{p, r \\ p \neq r}} a_p a_rL_0(p,r),$$
where $a_p = \frac{\log p } { \sqrt{p}} \Phi( p/X)$. By Mellin inversion, we get
\begin{align*}
S_{L_0} &= \sum_{\substack{p, r \\ p \neq r}} a_pa_r \sum_{\substack{a, c, e \\ a|c, c \leq C \\ (ce, pr) = 1}} \frac{\mu(a)}{a \varphi(ae)} \frac{\mu(c)}{\varphi(c)} W \left( \frac{c|p-r|}{eQ}\right) \\
&= \frac{1}{2\pi i} \int_{(-\varepsilon)} \hW(s) \sum_{\substack{p, r \\ p \neq r}} a_pa_r \sum_{\substack{a, c, e \\ a|c, c \leq C \\ (ce, pr) = 1}} \frac{\mu(a)}{a \varphi(ae)} \frac{\mu(c)}{\varphi(c)}  \left(\frac{c|p-r|}{eQ}\right)^{-s} \> ds.
\end{align*}
In order to separate the sums of $p$ and $r$, we need the following Mellin
transform
$$
|p - r|^{-s} = \frac{1}{2\pi i} \int_{(\delta)} \frac{\Gamma(1-s)\Gamma(z)}{\Gamma(1 -s + z)} \big ( p^{z - s} r^{-z} + r^{z - s} p^{-z} \big )
d z
$$
for $ p \neq r $, $ \delta > 0$ and $ \Rep(s) < 0$.
Note that the (absolute) convergence of the $z$ integral is ensured by the fact
that the product of Gamma factors decays like $|z|^{-1 + \Rep (s)}$.
Using the above identity we have
\begin{align*}
S_{L_0} &=  \frac{2}{(2\pi i)^2}\int_{(-\varepsilon)}\int_{(\delta)} \hW(s) Q^s\frac{\Gamma(1-s)\Gamma(z)}{\Gamma(1-s+z)}\\
& \qquad \cdot  \sum_{\substack{p, r \\ p \neq r}} \frac{a_p}{p^{s -z}} \frac{a_r}{r^z}   \sum_{\substack{c \leq C \\ (c, pr) = 1}} \frac{\mu(c)}{c^s\varphi(c)} \sum_{ a | c } \frac{\mu(a)}{a} \sum_{(e, pr) = 1} \frac{e^s}{\varphi(ae)} \> dz \> ds
\end{align*}
for $ \varepsilon >0$. Note that the sums over $a, c, p$ and $ r $ have only finitely many terms, so that there are no issues of convergence. The sum over $e$ is
$$ \sum_{(e, pr) = 1} \frac{e^{s}}{\varphi(ae)} = \frac{1}{\varphi(a)} \zeta(1-s) K(-s) B_{-s}(pr) R_{-s}(a) R_{-s}(pr)  $$
by Lemma \ref{lem:sumVarphid}, where the functions $K$, $B_s$ and $R_s$ are defined in Lemmas \ref{lem:primeSL} and \ref{lem:sumVarphid}. The sum over $a$ is
\begin{align*}
\sum_{  a | c } \frac{\mu(a)R_{-s}(a)}{a \varphi(a)} = \prod_{\ell | c} \left( 1 - \frac{R_{-s}(\ell)}{\ell(\ell - 1)}\right).
\end{align*}
Hence, we deduce
\begin{align*}
S_{L_0} & = \frac{2}{(2\pi i)^2}\int_{(-2\varepsilon)} \int_{(\varepsilon)}\hW(s) Q^s\zeta(1-s) K(-s) \frac{\Gamma(1-s)\Gamma(z)}{\Gamma(1-s+z)}  \\
& \qquad \cdot \sum_{\substack{p, r \\ p \neq r}} \frac{a_p}{p^{s -z}} \frac{a_r}{r^z} B_{-s}(pr) R_{-s}(pr)\sum_{\substack{c \leq C \\ (c, pr) = 1}} \frac{\mu(c)}{c^s\varphi(c)} \prod_{\ell | c} \left( 1 - \frac{R_{-s}(\ell)}{\ell(\ell - 1)}\right)  \> dz \> ds.
\end{align*}
We can remove the condition $(pr, c) = 1$ with an additional error $ O( C^{2 \varepsilon}Q^{- 2 \varepsilon} X^{1/2+3\varepsilon})$. The double sum over primes $p$ and $r$ is
\begin{align*}
\sum_{  p \neq r } & \frac{a_p}{p^{s -z}} \frac{a_r}{r^z} B_{-s}(p) B_{-s}(r) R_{-s}(p)R_{-s}(r) \\
 =&  \sum_{ p } \frac{a_p}{p^{s -z}} B_{-s}(p)  R_{-s}(p)  \cdot  \sum_r \frac{a_r}{r^z}  B_{-s}(r)  R_{-s}(r)   - \sum_p   \frac{a_p^2 }{p^s}  B_{-s}(p )^2 R_{-s}(p )^2 \\
 = & \left(\hP ( \tfrac12 -s+z ) X^{ 1/2-s+z} + O( X^{4 \varepsilon} \log ( 2+|s-z|) ) \right) \\
  & \times \left( \hP ( \tfrac12 - z ) X^{ 1/2- z} + O( X^\varepsilon  \log ( 2+|z|) ) \right) + O( X^\varepsilon )
\end{align*}
by Lemma \ref{lem:primeSL}. The error terms only contribute $O( Q^{-2 \varepsilon} C^{2 \varepsilon} X^{ 1/2 +4 \varepsilon}  )$ to $S_{L_0}$, so that
\begin{align*}
S_{L_0}   = &   \frac{2}{(2\pi i)^2}\int_{(-2\varepsilon)} \int_{(\varepsilon)}\hW(s) Q^s\zeta(1-s) K(-s) \frac{\Gamma(1-s)\Gamma(z)}{\Gamma(1-s+z)} \hP ( \tfrac12 -s+z )\hP ( \tfrac12 - z ) X^{1-s}   \\
 & \quad \cdot  \sum_{c \leq C} \frac{\mu(c)}{c^s\varphi(c)} \prod_{\ell | c} \left( 1 - \frac{R_{-s}(\ell)}{\ell(\ell - 1)}\right)
 \> dz \> ds  + O( Q^{-2 \varepsilon} C^{2 \varepsilon} X^{ 1/2 +4 \varepsilon}  ).
\end{align*}

To evaluate the integral, we consider two cases.
\\

{\bf Case 1:} $ X = Q^{\alpha}$, where $1 < \alpha < 2.$ In this case, we shift the contour of $s$ to $\Rep(s) = 1+ \varepsilon$ and get
\begin{align*}
S_{L_0} & = -( \mathrm{ Residue ~at~}  s=0  ) - ( \mathrm{ Residue ~at~}  s=1  ) \\
& \quad + \frac{2}{(2\pi i)^2}\int_{(1+\varepsilon )} \int_{(\varepsilon)}\hW(s) Q^s\zeta(1-s) K(-s) \frac{\Gamma(1-s)\Gamma(z)}{\Gamma(1-s+z)} \hP ( \tfrac12 -s+z )\hP ( \tfrac12 - z ) X^{1-s} \\
 & \qquad \cdot  \sum_{c \leq C} \frac{\mu(c)}{c^s\varphi(c)} \prod_{\ell | c} \left( 1 - \frac{R_{-s}(\ell)}{\ell(\ell - 1)}\right)
 \> dz \> ds  + O( Q^{-2 \varepsilon} C^{2 \varepsilon} X^{ 1/2 +4 \varepsilon}  )\\
&= -( \mathrm{ Residue ~at~}  s=0  ) - ( \mathrm{ Residue ~at~}  s=1  )  + O( Q^{1+\varepsilon} X^{ - \varepsilon}+ Q^{-2 \varepsilon} C^{2 \varepsilon} X^{ 1/2 +4 \varepsilon}  ).
\end{align*}
Three functions in the integrand have poles at $s=0 $ or $ s=1$. $\zeta(1-s)$ has a simple pole at $0$, $\Gamma(1-s)$ has a simple pole at $ s=1$ and $K(-s)$ has a simple pole at $s=1$, since
$$ K(-s) = \prod_{\ell} \left(1 + \frac{1}{(\ell -1)\ell^{1 -s}}\right) = \zeta(2-s) \prod_{\ell} \left(1 + \frac{1}{(\ell -1)\ell^{2 -s}} - \frac{1}{(\ell -1)\ell^{3-2s}}\right).$$
Hence, the residue at the simple pole $s=0$ is
\begin{align*}
     &  - \frac{1}{\pi i } \int_{(\varepsilon)}\hW(0 )   K(0) \frac{ \Gamma(z)}{\Gamma(1 +z)} \hP ( \tfrac12  +z )\hP ( \tfrac12 - z ) X
      \sum_{c \leq C} \frac{\mu(c)}{ \varphi(c)} \prod_{\ell | c} \left( 1 - \frac{R_{0}(\ell)}{\ell(\ell - 1)}\right)     dz\\
& =  - \frac{1}{\pi i }   \int_{(\varepsilon)}  \frac1z     \hP ( \tfrac12  +z )\hP ( \tfrac12 - z )     dz \cdot       \hW(0 )   K(0)    X
      \sum_{c \leq C} \frac{\mu(c)R_0 (c) }{ \varphi(c)}   \\
& = -  \hP  ( \tfrac 12  )  ^2        \hW(0 )   K(0)    X
      \sum_{c \leq C} \frac{\mu(c)R_0 (c) }{ \varphi(c)}   .
\end{align*}
 The residue at the double pole $s=1$ is
\begin{align*}
  & \frac{1}{ \pi  i }  \int_{(\varepsilon)} \hP ( -\tfrac12 +z )\hP ( \tfrac12 - z )      dz \cdot
   \zeta(0) \hW(1 ) Q ( \log \frac{Q}{X} +O(1))       \sum_{c \leq C} \frac{\mu(c)}{c \varphi(c)} \prod_{\ell | c} \left( 1 - \frac{R_{-1}(\ell)}{\ell(\ell - 1)}\right)     \\
 & = -\frac{1}{ 2\pi  i }  \int_{(1/2)} \hP ( -\tfrac12 +z )\hP ( \tfrac12 - z )      dz \cdot
   A_0  \hW(1 ) Q  \log \frac{Q}{X} + O(Q) \\
   &= -\frac{1}{ 2\pi    }  \int_{-\infty}^\infty  | \hP ( it  ) |^2    dt   \cdot A_0  \hW(1 ) Q \log \frac{Q}{X}  + O(Q)     ,
\end{align*}
where
\begin{align*}
A_0  = \sum_{c  } \frac{\mu(c)}{c \varphi(c)} \prod_{\ell | c} \left( 1 - \frac{R_{-1}(\ell)}{\ell(\ell - 1)}\right)   = \prod_p \left( 1- \frac{1}{p^2} - \frac{ 1}{ p^3} \right).
\end{align*}
 By \eqref{def:valueA} and the above, the residue at $s=1 $ is
$$ -\frac{1}{ 2\pi    }  \int_{-\infty}^\infty  | \hP ( it  ) |^2    dt\cdot
   A Q ( \log \frac{Q}{X} +O(1)) .$$
Combining all, we get
\begin{multline} \label{eqn:1}
 S_{L_0} =  \hP  ( \tfrac 12  ) ^2        \hW(0 )   K(0)    X        \sum_{c \leq C} \frac{\mu(c)R_0 (c) }{ \varphi(c)}
  + \frac{1}{ 2 \pi    }  \int_{-\infty}^\infty  | \hP ( it  ) |^2    dt\cdot
   A Q  \log \frac{Q}{X}    \\
    +O(Q + Q^{- 2 \epsilon} C^{2 \epsilon} X^{1/2 + 4 \epsilon} )
\end{multline}
for $ 1  < \alpha < 2 $ with $ X= Q^\alpha$. Note that the first term  in $S_{L_0}$ is cancelled with the main term of $ S_{U_0}$ in Equation (\ref{eqn:sumU}).
\\

{\bf Case 2:} $X = Q^{\alpha},$ where $0 \leq \alpha \leq 1.$ In this case, we shift the contour of $s$ to $\Rep(s) = \varepsilon$ and get
\begin{align*}
S_{L_0} & = -( \mathrm{ Residue ~at~}  s=0  )  \\
& \quad + \frac{2}{(2\pi i)^2}\int_{(\varepsilon )} \int_{(\varepsilon)}\hW(s) Q^s\zeta(1-s) K(-s) \frac{\Gamma(1-s)\Gamma(z)}{\Gamma(1-s+z)} \hP ( \tfrac12 -s+z )\hP ( \tfrac12 - z ) X^{1-s} \\
 & \qquad \cdot  \sum_{c \leq C} \frac{\mu(c)}{c^s\varphi(c)} \prod_{\ell | c} \left( 1 - \frac{R_{-s}(\ell)}{\ell(\ell - 1)}\right)
 \> dz \> ds  + O( Q^{-2 \varepsilon} C^{2 \varepsilon} X^{ 1/2 +4 \varepsilon}  )\\
&= -( \mathrm{ Residue ~at~}  s=0  )  + O( Q^{\varepsilon} X^{1 - \varepsilon}+ Q^{-2 \varepsilon} C^{2 \varepsilon} X^{ 1/2 +4 \varepsilon}  ).
\end{align*}
Since $0 \leq \alpha \leq 1,$ we obtain that
$$ Q^{\varepsilon} X^{1 - \varepsilon} = Q^{\alpha + (1 - \alpha) \varepsilon} \ll Q.$$
By the same argument as in Case 1, we get that
$$ S_{L_0} = \hP  ( \tfrac 12  )  ^2   \hW(0 )   K(0) X\sum_{c \leq C} \frac{\mu(c)R_0 (c) }{ \varphi(c)} + O( Q +  Q^{-2 \varepsilon} C^{2 \varepsilon} X^{ 1/2 +4 \varepsilon}  ),$$
and the first term is cancelled with the main term of $S_{U_0}.$
\\

The contribution from $L_E(p,r)$ is small as seen by the following lemma.
\begin{lemma} We have
\begin{equation} \label{eqn:2}
 S_{L_E} := \sum_{\substack{p, r \\ p \neq r}} \frac{\log p \log r \  \Phi (p/X) \Phi ( r/X) }{\sqrt{pr}}  L_E(p,r) \ll \frac{X^{1+ \varepsilon}C^{1 + \varepsilon}}{Q} .
\end{equation}
\end{lemma}
\begin{proof} Let $a_p$ be defined as in (\ref{def:ap}).  We have that
\begin{align*}
S_{L_E} &= \sum_{\substack{a, c, e \\ ac \leq C} }  \frac{\mu(a)\mu(ac)}{a\varphi(ac) \varphi(ae)} \sum_{\substack{\Psi \shortmod{ae} \\ \Psi \neq \Psi_0} }  \sum_{\substack{p, r \\ p \neq r,  (c, pr) = 1} } a_p a_r \Psi(p) \overline{\Psi}(r) W \left(\frac{ac |p - r|}{Qe}\right).
\end{align*}
Since $W$ is supported in $(1, 2),$ $\frac{ac|p-r|}{Qe} \geq 1$ and $e \leq \frac{ac|p-r|}{Q} \leq \frac{acX}{Q}.$
  Proceeding similarly to $S_{L_0}$ we obtain
\begin{align*}
S_{L_E}  & = \frac{1}{2 \pi i }  \int_{( -  \varepsilon)}\hW ( s)  Q^s  \sum_{\substack{a, c, e \\ ac \leq C , e \leq acX/ Q } }  \frac{\mu(a)\mu(ac)e^s }{a^{1+s} c^s \varphi(ac) \varphi(ae)} \sum_{\substack{\Psi \shortmod{ae} \\ \Psi \neq \Psi_0} }  \sum_{\substack{p, r \\ p \neq r \\ (c, p r ) = 1} } a_p a_r \Psi(p) \overline{\Psi}(r)   |p - r|^{-s} ds
\\
&= \frac{2 }{(2 \pi i)^2  }  \int_{( -  \varepsilon)}    \int_{( \varepsilon )}    \hW ( s)  Q^s \frac{\Gamma(1-s)\Gamma(z)}{\Gamma(1 -s + z)}   \sum_{\substack{a, c, e \\ ac \leq C , e \leq acX/ Q } }  \frac{\mu(a)\mu(ac)e^s }{a^{1+s} c^s \varphi(ac) \varphi(ae)} \times
\\ & \quad \sum_{\substack{\Psi \shortmod{ae} \\ \Psi \neq \Psi_0} }  \sum_{\substack{p, r \\ p \neq r \\ (c, pr ) = 1\\} } a_p a_r \Psi(p) \overline{\Psi}(r)  p^{z - s} r^{-z}        dz  ds  .
\end{align*}
The double sum over $p$ and $r$ is
\begin{align*}
&  \sum_{\substack{p, r \\ p \neq r \\ (c, pr ) = 1\\} } a_p a_r \Psi(p) \overline{\Psi}(r)  p^{z - s} r^{-z}     \\
&=  \sum_{\substack{p \\ (c,p) = 1}} \frac{\log p  \Psi(p) \Phi (p/X)  }{p^{1/2+s-z}}   \sum_{\substack{r \\ (c,r) =1}} \frac{ \log r     \overline{\Psi}(r) \Phi (r/X) } { r^{1/2+z}}    -  \sum_{\substack{p \\ (aec, p ) = 1\\} }  \frac{ (\log p)^2  \Phi^2 (p/X) }{p^{1+s}}
\end{align*}
and  bounded by $X^\varepsilon$ assuming GRH. The lemma easily follows from this bound.
\end{proof}

\subsection{Conclusion of the proof of Proposition \ref{prop:1}}

In the beginning of Section 3, we have shown that the sum $S$ splits into
$$
S = S_D + S_N
$$
with $S_D$ the diagonal terms and $S_N$ the off-diagonal terms.
In Section 3.1 we have shown that the diagonal terms $S_D$ contribute
$$
S_D \sim \frac{A}{2\pi} Q \log X \int_{-\infty}^{\infty} |\hP
(i x)|^2 d x
$$
In Sections 3.2--3.4 we have shown that $S_N = S_U+S_L$  is at most $O(Q)$ if $X = Q^{\alpha}$ with $0 \leq \alpha \leq 1$ and that
if $X = Q^{\alpha}$ with $1 < \alpha < 2$ then
$S_N$ is
$$
S_N \sim \frac{A}{2 \pi} Q \log  ( Q / X  ) \int_{-\infty}^{\infty}
|\hP (i x)|^2 dx,
$$
by (\ref{eqn:sumU}),
(\ref{eqn:1}) and
(\ref{eqn:2}), and choosing $C = Q^{\varepsilon}$. Combining the above estimates we conclude that
$$
S \sim  f(\alpha) \tfrac{A}{2\pi} Q \log Q \int_{-\infty}^{\infty} |\hP (ix)|^2 d x = f(\alpha) N_\Phi (Q)
$$
for $ 0 \leq \alpha < 2$, where $f(\alpha)$ is defined in \eqref{eqn:f alpha}. This gives the desired estimate.

\section{Proof of Theorem \ref{thm:1}}

Recall that
$$
F_{\Phi}(Q^{\alpha} ; W) :=  \frac{1}{N_\Phi (Q)} \sum_{q} \frac{W(q/Q)}{\varphi(q)}
\sumstar_{\chi \shortmod{q}} \bigg | \sum_{\gamma_\chi}
\hP (i \gamma_\chi) X^{i \gamma_\chi} \bigg |^{2}.
$$
Since $W$ is supported in $(1,2)$, there is no primitive character in the sum over $\chi$. Then by the Cauchy-Schwarz inequality, we have
\begin{align*}
  F_\Phi (Q^\alpha ; W) = M_1 + M_2 + O( \sqrt{M_1 M_2 } ) ,
\end{align*}
where
$$ M_1  := \frac{1}{N_\Phi (Q)} \sum_{q} \frac{W(q/Q)}{\varphi(q)}
\sumstar_{\chi \shortmod{q}} \bigg |  \sum_n \frac{ \Lambda(n) \chi(n)  \Phi(n/X) }{\sqrt{n}}   \bigg |^{2}, $$
and
$$ M_2  := \frac{1}{N_\Phi (Q)} \sum_{q} \frac{W(q/Q)}{\varphi(q)}
\sumstar_{\chi \shortmod{q}} \bigg | \Phi \left(X^{-1}\right) \log \frac{q}{ \pi} \bigg |^{2} .$$
By Proposition \ref{prop:1}, $ M_1 \sim f(\alpha) $ for $ |\alpha | \leq 2- \epsilon$. Also by a partial summation and (\ref{equation2}) we have
\begin{align*}
  M_2 \sim & \ \frac{1}{N_\Phi (Q)} \sum_{q} \frac{W(q/Q)}{\varphi(q)} \varphi^{*}(q) \Phi (X^{-1})^2 \log^2 Q \\
    \sim & \ \Phi(Q^{-|\alpha|} )^2 \log Q \left( \frac{1}{ 2 \pi }  \int_{-\infty}^\infty \left| \hP ( ix ) \right|^2 dx    \right)^{-1}.
\end{align*}
Therefore, we have
\begin{align*}
  F_\Phi (Q^\alpha; W) =& \ (1+o(1)) \left(  f(\alpha) + \Phi(Q^{-|\alpha|} )^2 \log Q \left( \frac{1}{ 2 \pi }  \int_{-\infty}^\infty \left| \hP ( ix ) \right|^2 dx    \right)^{-1}  \right) \\
  & + O(  \Phi( Q^{- |\alpha|} )  \sqrt{  f(\alpha) \log Q } ).
\end{align*}

\section{Proof of Theorem \ref{cor:1}}\label{proof of theorem 1}

We reproduce here the argument from \"{O}zl\"{u}k's paper \cite{Ozluk}.
First we need a lemma.

\begin{lemma} Assume GRH. If $1 < \alpha < 2$ is fixed, and the function $\Phi$ satisfies $ \Phi( x) = \Phi( x^{-1})$, then
\begin{multline*}
\frac{1}{N_\Phi (Q)} \sum_{q} \frac{W(q/Q)}{\varphi(q)} \sumstar_{\chi \shortmod{q}} \sum_{\gamma_\chi,\gamma'_\chi}
\bigg ( \frac{\sin (\alpha / 2(\gamma_\chi - \gamma'_\chi) \log Q)}{\alpha / 2
(\gamma_\chi - \gamma'_\chi) \log Q} \bigg )^{2} \hP
(i \gamma_\chi) \hP (i\gamma'_\chi )
\sim \bigg (1 + \frac{1}{3\alpha^2} \bigg )  .
\end{multline*}
\end{lemma}
\begin{proof}
We follow the argument given in \cite{Montgomery}. Let
$$
r(u) = \bigg ( \frac{\sin \pi \alpha u}{\pi \alpha u} \bigg )^2,
$$
and we use the identity
\begin{multline} \label{plugin}
\frac{1}{N_\Phi (Q)} \sum_{q} \frac{W(q/Q)}{\varphi(q)}
\sumstar_{\chi \shortmod{q}} \sum_{\gamma_\chi, \gamma'_\chi}
r \bigg ( \frac{(\gamma_\chi - \gamma'_\chi) \log Q}{2\pi} \bigg )
\hP (i \gamma_\chi)  \hP (i \gamma'_\chi) =
\int_{-\infty}^{\infty} F_{\Phi}(Q^\beta ; W) \tilde{r}(\beta) d \beta
\end{multline}
where $\tilde{r}(\beta)$ is the  Fourier  transform of $r$
defined as $ \tilde{r}(\beta) = \int_{-\infty}^{\infty} r(t) e^{-2\pi i \beta t} dt$. In this case
$$
\tilde{r}(\beta) = \begin{cases}
 (\alpha - |\beta|)/\alpha^2  & \text{ if } |\beta| < \alpha , \\
0 & \text{ otherwise}.
\end{cases}
$$
We plug in $F_{\Phi}(Q^{\beta};W)$ from Theorem \ref{thm:1} to the right-hand side of
(\ref{plugin}), obtaining that the right-hand side of (\ref{plugin}) is
$$
(1 + o(1))   \int_{-\alpha}^{\alpha} \left(  f(\beta) + \Phi(Q^{-|\beta|} )^2 \log Q \left( \frac{1}{ 2 \pi }  \int_{-\infty}^\infty \left| \hP ( ix ) \right|^2 dx    \right)^{-1}  \right)  \tilde{r}(\beta) d\beta.
$$
 For $ 1< \alpha < 2 $, we have
\begin{align*}
\int_{-\alpha}^{\alpha} f(\beta) \tilde{r}(\beta) d\beta & =
\frac{2}{\alpha^2} \int_{0}^{1} \beta \cdot (\alpha - \beta) d\beta
+ \frac{2}{\alpha^2} \int_{1}^{\alpha} (\alpha - \beta) d\beta \\
&  = 1 + \frac{1}{3\alpha^2} - \frac{1}{ \alpha},
\end{align*}
and
\begin{align*}
 \log Q &\left( \frac{1}{ 2 \pi }  \int_{-\infty}^\infty \left| \hP ( ix ) \right|^2 dx    \right)^{-1}  \cdot \int_{-\alpha}^{\alpha}   \Phi(Q^{-|\beta|} )^2  \tilde{r}(\beta) d\beta  \\
  & \sim  \log Q \left( \frac{1}{ 2 \pi }  \int_{-\infty}^\infty \left| \hP ( ix ) \right|^2 dx    \right)^{-1}  \cdot \frac{2}{\alpha^2} \int_0^1  \Phi(Q^{-\beta } )^2  (\alpha - \beta )  d\beta \\
  & \sim \frac{2}{\alpha} \int_0^{\log Q}  \Phi( e^{-u}  )^2   du   \left( \frac{1}{ 2 \pi }  \int_{-\infty}^\infty \left| \hP ( ix ) \right|^2 dx    \right)^{-1} \\
  & \sim \frac{2}{\alpha} \int_0^\infty  \Phi( e^{-u}  )^2   du   \left( \frac{1}{ 2 \pi }  \int_{-\infty}^\infty \left| \hP ( ix ) \right|^2 dx    \right)^{-1} \\
  & = \frac{1}{\alpha}.
\end{align*}
The last equality is obtained by Plancherel's theorem for Mellin transform in the form
$$ \frac{1}{ 2 \pi }  \int_{-\infty}^\infty \left| \hP ( ix ) \right|^2 dx = \int_{-\infty}^\infty  \Phi( e^{-u}  )^2   du =  \int_{-\infty}^\infty  \Phi( e^{-|u|}  )^2   du  $$
and the fact that the function $\Phi$ satisfies $ \Phi( x) = \Phi( x^{-1})$.

\end{proof}

\begin{proof}[Proof of Theorem \ref{cor:1}]
Pick $ \hP (s) = ((e^{s} - e^{-s})/2  s)^2$ so that $ \hP ( i \gamma)  = (\sin \gamma / \gamma)^2$. We need to check that this choice is possible, that
is, that $\Phi$ is real and compactly supported in $(a,b)$ for some $a,b > 0$.
Indeed, by Mellin inversion we have
\begin{align*}
\Phi(x) & =   \frac{1}{2\pi i} \int_{(c)} \bigg ( \frac{e^s - e^{-s}}{2  s} \bigg )^2 \cdot x^{-s} ds\\
 & =  \begin{cases}
    \frac12 - \frac14 \log x  & \text{ for } 1 \leq x \leq e^2  , \\
\frac12+ \frac14 \log x   & \text{ for }  e^{-2} \leq x \leq 1, \\
0 & \text{ otherwise},
\end{cases}
\end{align*}
so it satisfies the required conditions. Note that $\Phi$ satisfies $\Phi(x) = \Phi(x^{-1})$.

Let $m_{\rho}$ be the multiplicity of the zero $\rho = \tfrac 12 + i\gamma$.
We count zeros according to multiplicity. In particular,
$$
\sum_{\gamma_{\chi}} m_{\rho} \hP (i \gamma_{\chi})^2
 = \sum_{\substack{\gamma_\chi,\gamma'_\chi  \\ \gamma_\chi = \gamma'_\chi}}
 \hP (i \gamma_\chi)  \hP (i \gamma'_\chi)
$$
because on both sides a given
zero is counted with weight $m^2_\rho \hP (i\gamma_\chi )^2$. We have
\begin{align*}
\sum_{\substack{\gamma_{\chi} \\ \text{simple}}} \hP (i \gamma_{\chi})^2
& \geq \sum_{\gamma_\chi} (2 - m_\rho)\hP (i \gamma_\chi)^2 \\
& \geq 2 \sum_{\gamma_\chi} \hP (i \gamma_\chi)^2
- \sum_{\gamma_\chi, \gamma'_\chi} \bigg ( \frac{\sin \alpha / 2 (\gamma_\chi
-\gamma_\chi') \log Q}{\alpha/2 (\gamma_\chi - \gamma'_\chi) \log Q} \bigg )^2
\hP(i \gamma_\chi)\hP (i \gamma'_\chi).
\end{align*}
Hence
\begin{align*}
\sum_{q} & \frac{W(q/Q)}{\varphi(q)} \sumstar_{\chi \shortmod{q}}
\sum_{\substack{\gamma_\chi \\ \text{simple}}} \hP( i \gamma_\chi)^2  \\
& \geq 2 \sum_q \frac{W(q/Q)}{\varphi(q)} \sumstar_{\chi \shortmod{q}}
\sum_{\gamma_\chi} \hP(i \gamma_\chi)^2 \\
& - \sum_{q} \frac{W(q/Q)}{\varphi(q)} \sumstar_{\chi \shortmod{q}}
\sum_{\gamma_\chi, \gamma'_\chi} \bigg ( \frac{\sin (\alpha/2 (\gamma_\chi
- \gamma'_\chi) \log Q)}{\alpha / 2 (\gamma_\chi - \gamma'_\chi) \log Q}
\bigg )^2 \hP(i \gamma_\chi) \hP (i \gamma'_\chi).
\end{align*}
We take $\alpha = 2-\delta$ with $\delta > 0$
in the previous lemma and observe that
\begin{multline*}
\sum_{q} \frac{W(q/Q)}{\varphi(q)} \sumstar_{\chi \shortmod{q}} \sum_{\gamma_\chi,\gamma'_\chi}
\bigg ( \frac{\sin (\alpha / 2(\gamma_\chi - \gamma'_\chi) \log Q)}{\alpha / 2
(\gamma_\chi - \gamma'_\chi) \log Q} \bigg )^{2} \hP
(i \gamma_\chi)\hP(i\gamma'_\chi )
\leq \bigg ( \frac{13}{12} + \varepsilon \bigg )
N_{\Phi}(Q)
\end{multline*}
with some $\varepsilon \rightarrow 0$ as $\delta \rightarrow 0^{+}$.
Combining the above two equations and using the fact that
$$
\sum_{q} \frac{W(q/Q)}{\varphi(q)} \sumstar_{\chi \shortmod{q}}
\sum_{\gamma_\chi} \hP(i \gamma_\chi)^2
= N_{\Phi}(Q)
$$
we prove Theorem \ref{cor:1}.
\end{proof}

\section{Discussion of \"{O}zl\"{u}k's result}\label{8688}
In this section we explain why heuristically one
expects \"{O}zl\"{u}k's result to provide a proportion of at most 86\%
simple zeros.
It is reasonable to suppose that as $t \rightarrow \infty$, there exists a $\kappa$ such that
$$
\sum_{q\leq t} \frac{1}{\varphi(q)}   {\sumstar_{\chi \shortmod{q}}} \sum_{  \substack{\gamma_{\chi}\\ \text{simple}} } |\hP(i \gamma_{\chi})|^2 \sim \kappa \frac{t \log t}{2 \pi} \int_{-\infty}^{\infty} |\hP (i x)|^2 dx .
$$
\"{O}zl\"{u}k proves that
\begin{equation*}
  \sum_{q \leq Q} \frac{1}{\varphi(q)} {\sum_{\chi \shortmod{q}}}
\sum_{\substack{\gamma_{\chi}\\ \text{simple}} } |\hP(i \gamma_{\chi})|^2 \gtrsim
 \frac{11}{12}  \frac{Q \log Q}{2 \pi} \int_{-\infty}^{\infty} |\hP(i x)|^2 dx .
\end{equation*}
We re-write the left-hand side as follows
\begin{equation*}
 \sum_{q \leq Q} \frac{1}{\varphi(q)} {\sum_{\chi \shortmod{q}}} \sum_{  \substack{\gamma_{\chi}\\ \text{simple}} } |\hP(i \gamma_{\chi})|^2 \\
  = \sum_{q \leq Q} \frac{1}{\varphi(q)} \sum_{d | q }  {\sumstar_{\chi^* \shortmod{d}}} \sum_{ \substack{\gamma_{\chi^*}\\ \text{simple}} } |\hP(i \gamma_{\chi^*})|^2 ,
\end{equation*}
where $\chi^*$ is the primitive character inducing $ \chi$. Note that the nontrivial zeros of $L(s, \chi)$ and $L(s,\chi^*)$ coincide.
Therefore, we get
\begin{align*}
 \sum_{q \leq Q} \frac{1}{\varphi(q)} {\sum_{\chi \shortmod{q}}}
\sum_{\substack{\gamma_{\chi}\\ \text{simple}} } |\hP(i \gamma_{\chi})|^2   & = \sum_{dq \leq Q} \frac{1}{\varphi(dq)}   {\sumstar_{\chi \shortmod{q}}} \sum_{  \substack{\gamma_{\chi}\\ \text{simple}} } |\hP (i \gamma_{\chi})|^2  \\
  & \leq \sum_{d \leq Q}  \frac{1}{ \varphi(d)} \sum_{q\leq Q/d} \frac{1}{\varphi(q)}   {\sumstar_{\chi \shortmod{q}}} \sum_{  \substack{\gamma_{\chi}\\ \text{simple}} } |\hP(i \gamma_{\chi})|^2  \\
&  \lesssim  \sum_{d \leq Q}  \frac{1}{ \varphi(d)}  \kappa \frac{Q}{d} \log \frac{Q}{d}  \frac{1}{2 \pi} \int_{-\infty}^{\infty} |\hP(i x)|^2 dx   \\
   & \sim \kappa  \sum_{d=1}^\infty \frac{1}{ d \varphi(d)}   \frac{ Q \log Q  }{ 2 \pi} \int_{-\infty}^{\infty} |\hP(i x)|^2 dx  .
\end{align*}
It thus follows that
$$ \kappa \geq \frac{11}{12} \left( \sum_{d=1}^\infty \frac{1}{ d \varphi(d)}  \right)^{-1},$$
or equivalently
$$ \frac{1}{ N'_\Phi (Q)} \sum_{q\leq Q} \frac{1}{\varphi(q)}   {\sumstar_{\chi \shortmod{q}}} \sum_{  \substack{\gamma_{\chi}\\ \text{simple}} } |\hP(i \gamma_{\chi})|^2 \geq   \frac{11}{12} \left( \sum_{d=1}^\infty \frac{1}{ d \varphi(d)}  \right)^{-1} A_0^{-1}, $$
where
\begin{align*}
 N'_\Phi (Q) & := \sum_{q\leq Q} \frac{1}{\varphi(q)}   {\sumstar_{\chi \shortmod{q}}} \sum_{  \gamma_{\chi}  } |\hP(i \gamma_{\chi})|^2 \\
 & \sim \frac{1}{ 2 \pi} \log Q \int_{-\infty}^{\infty} |\hP(i x)|^2 dx \sum_{q \leq Q} \frac{ \varphi^*(q)}{\varphi (q)} \\
 & \sim A_0  \frac{Q\log Q}{ 2 \pi}   \int_{-\infty}^{\infty} |\hP(i x)|^2 dx
 \end{align*}
and
$$ A_0 = \prod_p \left( 1- \frac{1}{p^2} - \frac{1}{p^3} \right). $$
Therefore from \"{O}zl\"{u}k's work we obtain a proportion of
$$  \frac{11}{12} \left( \sum_{d=1}^\infty \frac{1}{ d \varphi(d)}  \right)^{-1} A_0^{-1}   \approx 0.86883781 \ldots .$$

\section{Corollary \ref{colofprop1} and Barban-Davenport-Halberstam Theorem} \label{sec:BDH}
The prime number theorem in arithmetic progression asserts that 
$$ \psi(x; q, a)  := \sum_{\substack{n \leq x \\ n \equiv a \ {\rm mod} \ q}} \Lambda(n) \sim \frac{x}{\phi(q)} $$
 as $x \rightarrow \infty$ for fixed $q$. We are interested in the mean square of the error term in the prime number theorem for arithmetic progressions as $q$ varies and consider 
$$ M(x, Q) = \sum_{q \leq Q} \sum_{\substack{a = 1 \\ (a, q) = 1}}^q \left| \psi(x; q, a)  - \frac{x}{\phi(q)}\right|^2.$$
This was first studied by Barban \cite{Barban}, and independently by Davenport and Halberstam \cite{DavaHalb}. Bounds and asymptotic formulas of $M(x, Q)$ are usually referred to as the Barban-Davenport-Halberstam Theorem. Gallagher \cite{Gal} later refined their results by showing that
$$ M(x, Q) \ll xQ\log x$$
for $x(\log x)^{-A} \leq Q \leq x.$ This estimate is the best possible since Montgomery \cite{M:BDH} and Hooley \cite{Hooley:BDH} showed that 
\begin{equation}\label{MxQ}
M(x, Q) \sim xQ\log x
\end{equation}
 in the same range. Moreover, Montgomery proved \eqref{MxQ} for $x \leq Q$ and Hooley \cite{Hooley:BDHII} obtained \eqref{MxQ} on GRH for $x^{1/2 + \epsilon} \leq Q \leq x$.

The proof of the Barban-Davenport-Halberstam Theorem reduces essentially to finding bounds or asymptotic formulas for
\begin{equation} \label{sum:BDH}
 \sum_{q \leq Q} \frac{1}{\phi(q)} \sum_{\chi \neq \chi_0} \bigg| \sum_{n \leq x} \Lambda(n) \chi(n) \bigg|^2,
\end{equation}
since
$$
 M(x, Q)    \sim  \sum_{q \leq Q}\sum_{\substack{a = 1 \\ (a, q) = 1}}^q \bigg| \psi(x; q, a)  - \frac{\psi(x, \chi_0)}{\phi(q)}\bigg|^2    = \sum_{q \leq Q} \frac{1}{\phi(q)} \sum_{\chi \neq \chi_0} \bigg| \sum_{n \leq x} \Lambda(n) \chi(n) \bigg|^2
$$
using the fact that $\psi(x, \chi_0 ) := \sum_{n \leq x} \Lambda(n)\chi_0 (n) \sim x$ as $ x \to \infty$ for a principal character $ \chi_0$ and the orthogonality relations for characters.

In Corollary \ref{colofprop1},  we obtain on GRH an asymptotic formula for a sum similar to (\ref{sum:BDH}) in the range $Q \geq x^{1/2 + \epsilon}$. Our sum is taken over primitive characters instead of all characters. The main term in the asymptotic formula in Corollary \ref{colofprop1} is (after normalization) the same as the main term in the asymptotic formulas in Montgomery's and Hooley's results.

\section{Acknowledgments}

This work was initiated during the \textit{Arithmetic Statistics} MRC program at Snowbird.
We would like to thank Brian Conrey for his guidance throughout the project and for providing
us with many unpublished materials on the asymptotic large sieve. We would also like to thank Kannan Soundararajan  for pointing out the relation of our work to the Barban-Davenport-Halberstam Theorem. Finally, we would like to thank the referee for many valuable comments.

\end{document}